\documentclass[12pt]{article}
\usepackage{pslatex}
\usepackage{fancyhdr}
\usepackage{graphicx}
\usepackage{geometry}
\RequirePackage[latin1]{inputenc} \RequirePackage[T1]{fontenc}
\def\figurename{Figure}\makeatletter
\renewcommand{\fnum@figure}[1]{\figurename~\thefigure.}
\makeatother

\def\tablename{Table}
\makeatletter
\renewcommand{\fnum@table}[1]{\tablename~\thetable.}
\makeatother
\usepackage{color}
                [2000/03/29 v1.0m Standard LaTeX package]

\usepackage{amsmath}
\usepackage{amssymb}
\usepackage{amsfonts}
\usepackage{amsthm,amscd}

\newtheorem{theorem}{Theorem}[subsection]

\newtheorem{corollary}{Corollary}[subsection]
\newtheorem{proposition}{Proposition}[subsection]

\newtheorem{example}{Example}[subsection]

\theoremstyle{remark}
\newtheorem{remark}[theorem]{Remark}

\numberwithin{equation}{subsection}

\def\P{\mathbb P}
\def\R{\mathbb R}
\def\E{\mathbb E}

\def\Q{\mathbb Q}

\def\E{\mathbb E}
\def\N{\mathbb N}

\begin{document}

\title{Generalized delayed Black and Scholes Formula}

\author{Hubert Le Bi Golé \thanks{lebigole@gmail.com}\;\; and \; Auguste Aman \thanks{aman.auguste@ufhb.edu.ci/augusteaman5@yahoo.fr, corresponding author}\\
UFR Math\'{e}matiques et Informatique, Université Félix H. Boigny, Cocody,\;\;\;\;\;\;\;\;\;\;\\}

\date{}
\maketitle \thispagestyle{empty} \setcounter{page}{1}

\thispagestyle{fancy} \fancyhead{}
 \fancyfoot{}
\renewcommand{\headrulewidth}{0pt}

\begin{abstract}
The mean objective of this paper is to derive an explicit formula for a price of  an European option associated to the underlying delayed stock price which follows a linear differential equation with a general delay in the drift term. We use an equivalent martingale measure method based on Girsanov's property. Two of our model maintains the no-arbitrage property and the completeness of the market and can be considered as an extension some previous model introduced by Arriojas et al. in \cite{Aal}. The last one has a possible arbitrage property such that we can not obtain an unique price of  an European option associated.  
\end{abstract}

\vspace{.08in} \noindent \textbf{MSC}:Primary: 60F05, 60H15; Secondary: 60J30\\
\vspace{.08in} \noindent \textbf{Keywords}: Forward stochastic Volterra integral equations, Time delayed generators, Black and Scholes formula, Equivalent martingale measure, Option pricing.

\section{Introduction}
The continuous models have gained importance since the work of Bachelier \cite{Bachelier}. In this dynamic, the Black and Scholes model has given the interest to stochastic differential equations in financial mathematics modeling.
This modeling is one of the best known tools which allows, in theory, to evaluate an option with some given play-off, based on the underlying financial asset $S= (S_t, 0 \leq  t \leq  T)$ defined on some probability space $(\Omega,\mathcal{F},\P)$ and solution of the following Samuelson SDE: .
\begin{eqnarray}\label{bond}
 dS(t)= S(t)\left[\mu\ dt+\sigma\ dW(t)\right], \ \ S(0) = x,
\end{eqnarray}
where $\mu$ is the expected instantaneous rate of the risky asset, $\sigma$ is the volatility of the risky asset and $W=(W(t))_{t\geq 0}$ is a one-dimensional Brownian motion.

This evaluation is based on the notion of admissible portfolio strategy. More precisely, let $\Pi=(\varphi^0,\varphi)$ be a self-funded portfolio, constituted by risk-free assets and risky bond respectively. Setting by $V=(V(t))_{0\leq t\leq T}$ its value, we have
 \begin{eqnarray}\label{portfolio}
	dV(t)=\pi_B(t)dB(t)+\pi_S(t)dS(t),
\end{eqnarray}
where $B=(B(t), 0 \leq t \leq T)$ denotes risk-free asset and satisfies the following ordinary differential equation (ODE, in short). 
 
 \begin{eqnarray}\label{EDO}
dB(t)=r B(t) dt, \ \ B(0)= 1,
\end{eqnarray}
with $r$ the interest rate in the market. 

In this contexte the price of this option is $V(0)$ which is defined by
\begin{eqnarray*}
	V(0)=e^{-rT}\E^{*}(V(T)|S(0)=x),
\end{eqnarray*}
where $\E^{*}$ design an expectation which respect $\P^{*}$ named risk-neutral probability and defined by
\begin{eqnarray*}
	d\P^{*}=\xi(T)d\P,
\end{eqnarray*}
with 
\begin{eqnarray*}
	\xi(T)=\exp\left(-\frac{\mu-r}{\sigma}W(T)-\frac{1}{2}\left(\frac{\mu-r}{\sigma}\right)^2T\right).
\end{eqnarray*}

However, like any other model, the Black and Scholes model has some shortcomings. Indeed, in this model, the volatility of the risky bond is assumed to be constant, whereas an empirical study of market data showed it actually depends on time in an unpredictable way.

Next, in view of \eqref{bond}, the model assumes that the underlying asset follows geometric Brownian motion, meaning that the asset's price changes are normally distributed and independent of each other, this assumption may not hold in reality, as asset prices may exhibit jumps, volatility clusters, and other forms of non-normality such that such as skewness and kurtosis.

On the other hand, the emergence of a preference effect due to investors' risk aversion has led to a defect that undermines the aforementioned modeling in real data observations (stock prices, options, etc.), see Loewenstein \cite{Lo1,Lo2} for more details. In view of this, we assume that investors make a comparison of past and present opportunities by taking into account observed price trends or the satisfaction of past consumption.
Thus, Arriojas et al consider in \cite{Aal} the effect of the past in the determination of the fair price of a option. More specifically, they assume that the stock price satisfies stochastic functional differential equations (SFDEs) of the form: 
\begin{eqnarray*}
  	S(t)&=&\varphi(0)+\mu\int_{0}^{t}S(s-a)ds+\int_{0}^{t}g(S(s-b))S(s)dW(s),\ t\in[0,T]\nonumber \\
  S(t)&=&\varphi(t),\ t\in[-T,0]
  \end{eqnarray*}
  and 
  \begin{eqnarray*}
  	S(t)&=&\varphi(0)+\mu\int_{0}^{t}S(s-a)S(s)ds+\int_{0}^{t}g(S(s-b))S(s)dW(s),\ t\in[0,T]\nonumber \\
  S(t)&=&\varphi(t),\ t\in[-T,0],
  \end{eqnarray*}
  where $\mu, a$ and $b$ a positive constant and $g$ a continuous function.
Among other, they derive an explicit formula for the valuation of a European call option on a above given stock.

In this article, we establish the Black-Scholes formula with the stock whose price satisfies the following generalized stochastic functional differential equations (SFDEs) of the form: for $i=1,2$,
\begin{eqnarray}\label{AA}
	S(t)&=&\varphi(0)+\int_{0}^{t}f_i(s,S_s)ds+\int_{0}^{t}g(S(s-b))S(s)dW(s),\ t\in[0,T] \nonumber\\
  S(t)&=&\varphi(t),\ t\in[-T,0],
\end{eqnarray} 
with
\begin{eqnarray*}
f_1(t,S_t)=\int^{-a}_{-T}\mu(s+u)S(t+u)\alpha(du)\;\;\mbox{and}\;\; f_2(t,S_t)=\int^{-a}_{-T}\mu(s+u)S(t+u)S(t)\alpha(du),
\end{eqnarray*}
where, for a fix $a\in]0,T[$, $\alpha$ is a measure of probability on $([-T,-a],\mathcal{B}([-T,-a]))$, . 

The rest of this paper is organized as follows. In section 2, we study the existence and uniqueness result for SDE \eqref{AA}. The section 3 is devoted to derive an explicit formula for the valuation of a European call option on a given stock with respect SFDEs \eqref{AA}.

\section{General stochastic delay model for stock price}
In this paper, we  consider $W = \{W(t),\; t\geq 0 \}$ (rep. $\alpha$) a standard one dimensional Brownian motion (resp.  a probability measure) defined on $(\Omega, \mathcal{F},\P)$ (resp. on $([-T,-a],\mathcal{B}([-T,-a]))$). Let $(S(t))_{t\geq 0}$ denote a price of a bond in a financial market. We will assume that the process $(S(t))_{t\geq 0}$ satisfies the following general delayed stochastic differential equation (GDSDEs, in short): 
\begin{eqnarray}\label{GDSDE}
S(t)&=&\varphi(0)+\int_{0}^{t}f(s,S_s)ds+\int_{0}^{t}g(S(s-b))S(s)dW(s),\ t\in[0,T] \nonumber\\
  S(t)&=&\varphi(t),\ t\in[-T,0],
\end{eqnarray}
where, $b$ is positive constant and $S_t=(S(t+u))_{-T\leq u\leq 0}$ denotes all the past of $S$ until $t$.
We work also with  the following spaces and assumptions:
\begin{description}
\item $\bullet$ Let $\mathcal{C}([-T,0],\mathbb{R})$ is a Banach space of all continuous functions $\varphi : [-T,0] \rightarrow\mathbb{R}$ with the supremum norm 
\begin{eqnarray}\label{AA1}
\sup_{-T\leq s\leq 0}|\varphi(s)|^2.	
\end{eqnarray} 
\item $\bullet$ For a given process $\varphi\in \mathcal{C}([-T,0],\mathbb{R}) $, let $\mathcal{S}^{2}_{\varphi}(\R)$ denote the space of all predictable and almost surely continuous process $(\phi(t))_{-T\leq t\leq T}$ with values in $\R$ such that $\phi(t)=\varphi(t)$ a.s. for all $t\in [-T,0]$ and  $$\E\left[\sup_{0\leq s\leq T}|\phi(s)|^2\right]<+\infty. $$
\end{description}
\begin{description}
	\item $({\bf A1})$ $f:[0,T]\times \mathcal{C}([-T,0],\mathbb{R}^{*}_{+})\rightarrow \R$ such that
\begin{eqnarray*}
	|f(t,s^{1}_{t})-f(t,s_{t}^{2})|\leq K\|s_{t}^{1}-s_{t}^{2}\|.
\end{eqnarray*}
\item  $({\bf A2})$ $g:\mathbb{R}\longrightarrow \mathbb{R}$ is continuous.
\item $({\bf A3})$
 $\varphi:\Omega\longrightarrow\mathcal{C}([-T,0],\mathbb{R})$ is a ${\mathcal{F}}_{0}$-measurable and positive function.
\item $({\bf A4})$ for all $x\in\R^{*},\; g(x)\neq 0$.
\item $({\bf A5})$ the probability measure $\alpha$ is a uniform probability measure.
 \end{description}
 In view of Proposition 2.1.$(i)$, in \cite{CA} we have the following existence and uniqueness result.
\begin{theorem}\label{T01}
Assume assumptions $({\bf A1})$-$({\bf A3})$. Then SDE \eqref{GDSDE} has a unique solution. 
\end{theorem}
In the sequel, we will deal with these following two cases of $f$. For this let recall $\alpha$ be probability measure on $([-T,-a],\mathcal{B}([-T,-a]))$.

\subsection{Case 1: $\displaystyle f(t,s_t)=\int_{-T}^{-a}\mu(t+u)s(t+u)\alpha(du)$}
Let $f$ be defined by 
\begin{eqnarray}\label{F1}
f(t,s_t)=\int_{-T}^{-a}\mu(t+u)s(t+u)\alpha(du),	
\end{eqnarray}
where
\begin{eqnarray*} 
\begin{cases}
\mu(t)=\mu, \ t\in[-a,T] \\
  \mu(t)=0,\ t\in[-T,-a].
 \end{cases} 
 \end{eqnarray*}
In this situation, the generalized delayed SDE \eqref{GDSDE} becomes 
\begin{eqnarray}\label{T1}
\begin{cases}
S(t)=\varphi(0)+\mu\displaystyle\int_{0}^{t}\displaystyle\int_{-T}^{-a}\mu(s+u)S(s+u)\alpha(du)ds+\displaystyle\int_{0}^{t}g(S(s-b))S(s)dW(s), \ t\in[0,T] \\\\
  S(t)=\varphi(t),\ t\in[-T,0].
 \end{cases} 
 \end{eqnarray}
\begin{remark}
Let for example taking $\alpha=\delta_{-a}$ then 
\begin{eqnarray*}
\int_{-T}^{-a}\mu(t+u)S(t+u)\alpha(du)&=&\int_{-T}^{-a}\mu(t+u)S(t+u)\delta_{-a}(du)\\
	&=& \mu(t-a) S(t-a)
\end{eqnarray*}
and equation \eqref{T1} coincide with SDDE $(2)$ appeared in section 2.1 in \cite{Aal}.
\end{remark}
 We are now ready to derive our first result in this section.
\begin{theorem}\label{Theorem 1}
 Assume $({\bf A2})$-$({\bf A3})$. Then GDSDE  \eqref{T1} has a unique solution. Furthermore, let $(S(t))_{t\geq 0}$ be it solution. Then we have 
\begin{eqnarray}\label{SA}
S(t)=\varphi(0)+\mu\int_{-a}^{t-a}\alpha([-T\wedge (s-t),-a])S(s)ds+\int_{-a}^{t-a}g(S(s-b))S(s)dW(s),\;\; t\in [0,T].
\end{eqnarray}
\end{theorem}
\begin{proof}
In view of \eqref{F1}, it not difficult to prove that $f$ is $\mu$-Lipschitz. Then according to Theorem \ref{T01}, the delayed SDE \eqref{T1} admit a unique solution. It remain to derive \eqref{SA}. For this, let set  
\begin{eqnarray*}
	I(t)=\int_{0}^{t}\int_{-T}^{-a}\mu(s+u)S(s+u)\alpha(du)ds. 
\end{eqnarray*}
It follows from Fubini's theorem and the change of variable that 
\begin{eqnarray*}
I(t)&=&\int_{-T}^{-a}\left(\int_{u}^{t+u}\mu(r)S(r)dr\right)\alpha(du)\\
&=&\int_{-T}^{-a}\left(\int_{-T}^{t-a}\mu(r)S(r){\bf 1}_{[u,t+u]}(r)dr\right)\alpha(du)\\
&=&\int_{-T}^{0}\left(\int_{-T}^{t-a}\mu(r)S(r){\bf 1}_{[(r-t),\; r]}(u)dr\right)\alpha(du)\\
&=&\int_{-T}^{t-a}\mu(r)S(r)\left(\int_{-T}^{0}{\bf 1}_{[r-t,\; r]}(u)\alpha(du)\right)dr\\
&=&\int_{-T}^{t}\mu(r)S(r)\alpha\left([-T\vee(r-t),\; -a\wedge r])\right)dr\\
&=&\int_{-T}^{-a}\mu(r)S(r)\alpha\left([-T\vee(r-t),\; -a\wedge r])\right)dr+\int_{-a}^{t-a}\mu(r)S(r)\alpha\left([-T\vee(r-t),\; -a\wedge r])\right)dr.
\end{eqnarray*}
Since for all $r\in[-T,-a],\; \mu(r)=0$, we have
\begin{eqnarray*}
I(t)  =\int_{-a}^{t-a}\alpha\left([-T\vee(r-t),\; -a])\right)S(r)dr.
\end{eqnarray*}
Putting this in \eqref{T1} we obtain the result.
\end{proof}
\begin{remark}
Equation \eqref{SA} is the Volterra SDE. It has a unique solution which can not be determined explicitly in general case.
\end{remark}

\begin{corollary}
Assume $({\bf A2})$-$({\bf A5})$. Then SDE \eqref{SA} admit almost surely positive solution. Moreover, setting $l = \min(a,b,T-a)> 0$, we get, for all $t\in [0,T]$, 
\begin{eqnarray}\label{S1}
S(t)&=&\varphi(0)\exp\left(\int_{0}^{t}g(S(s-b))dW(s)+\frac{1}{2}\int_{0}^{t}g^2(S(s-b))ds\right)\\
 &&+\mu\left[\int_{0}^{t}\exp\left(\int_{s}^{t}g(S(r-b))dW(r)+\frac{1}{2}\int_{s}^{t}g^2(S(r-b))dr\right)\psi(s)ds\right]{\bf 1}_{[0,T-a]}(t)\nonumber\\
 &&+\mu\left[\int_{0}^{t}\exp\left(\int_{s}^{t}g(S(r-b))dW(r)+\frac{1}{2}\int_{s}^{t}g^2(S(r-b))dr\right)S(s)ds\right]{\bf 1}_{[T-a,T]}(t),\nonumber
 \end{eqnarray}
 where
 \begin{eqnarray*}
\psi(s)=\frac{1}{T-a}\sum_{k=1}^{p}\left(\int_{(k-1)l-a}^{s-a}S(r)dr\right){\bf 1}_{[(k-1)l,kl]}(s)+\frac{1}{T-a}\left(\int_{pl-a}^{s-a}S(u)\alpha(du)\right){\bf 1}_{[pl,T-a]}(s),
\end{eqnarray*}
with $p\in \N^{*}$ such that $pl\leq T-a <(p+1)l$
\end{corollary}
\begin{proof}
According to Theorem \ref{Theorem 1} and assumption $({\bf A5})$, equation \eqref{SA} will be transform to 
\begin{eqnarray}\label{S11}
S(t)&=&\varphi(0)+\frac{\mu}{T-a}\int_{-a}^{t-a}(T\wedge (t-s)-a)S(s)ds+\int_{0}^{t}g(S(s-b))S(s)dW(s)\nonumber\\
&=& \varphi(0)+\frac{\mu}{T-a}\left(\int_{-a}^{t-a}(t-s-a)S(s)ds\right){\bf1}_{[0,T-a]}(t)\nonumber\\
&&+\frac{\mu}{T-a}\left(\int_{-a}^{t-a}(T-a)S(s)ds\right){\bf1}_{[T-a,T]}(t)+\int_{0}^{t}g(S(s-b))S(s)dW(s) .
 \end{eqnarray}
For $t\in [0,l]$, it follows from \eqref{S11} that 
\begin{eqnarray*}
S(t)&=&\varphi(0)+ \frac{\mu}{T-a}\int_{-a}^{t-a}(t-s-a)\varphi(s)ds+\int_{0}^{t}g(\varphi(s-b))S(s)dW(s).
 \end{eqnarray*}
whose differential form is given by
 \begin{eqnarray*}
dS(t)=\frac{\mu}{T-a}\left(\int_{-a}^{t-a}\varphi(r)dr\right) dt+g(\varphi(t-b))S(t)dW(t),\;\; S(0)=\varphi(0).
 \end{eqnarray*}
Set $(S^{1}(t))_{0\leq t\leq l}$ its solution. Using the standard method of resolution to linear SDE, we obtain
\begin{eqnarray*} 
 S^1(t)&=&\varphi(0)\exp\left(\int_{0}^{t}g(\varphi(s-b))dW(s)+\frac{1}{2}\int_{0}^{t}g^2(\varphi(s-b))ds\right)\\
 &&+\int_{0}^{t}\exp\left(\int_{s}^{t}g(\varphi(r-b))dW(r)+\frac{1}{2}\int_{s}^{t}g^2(\varphi(r-b))dr\right)\psi(s)ds,
 \end{eqnarray*} 
where 
\begin{eqnarray*}
\psi^1(s)=\frac{\mu}{T-a}\int_{-a}^{s-a}\varphi(r)dr.
\end{eqnarray*}
 Therefore, since $\varphi$ is a positive function, the process $S^1$ is almost surely positive.
  
 Let now consider $S^2$ the solution of \eqref{T1} on $[l,2l]$.
 
 If $2l<T-a$, hence it follows from \eqref{S11} that for $t\in [l,2l]$  
\begin{eqnarray*}
S^2(t)&=& S^1(l)+\frac{\mu}{T-a}\int_{l-a}^{t-a}(t-s-a)S^2(s)ds+\int_{l}^{t}g(S^2(s-b))S^2(s)dW(s)\\
&=& S^1(l)+\frac{\mu}{T-a}\int_{l-a}^{t-a}(t-s-a)S^1(s)ds+\int_{l}^{t}g(S^1(s-b))S^2(s)dW(s).
 \end{eqnarray*}
 Using the similar argument as above we obtain for all $t\in [l,2l]$
\begin{eqnarray*} 
 S^2(t)&=& S^1(l)\exp\left(\int_{l}^{t}g(S^1(s-b))dW(s)+\frac{1}{2}\int_{l}^{t}g^2(S^1(s-b))ds\right)\\
 &&+\int_{l}^{t}\exp\left(\int_{s}^{t}g(S^1(r-b))dW(r)+\frac{1}{2}\int_{s}^{t}g^2(S^1(r-b))dr\right)\psi^2(s)ds,
 \end{eqnarray*}
 where 
\begin{eqnarray*}
\psi^2(s)=\frac{\mu}{T-a}\int_{l-a}^{s-a}S^1(r)dr.
\end{eqnarray*}
Since $S^1$ is positive, we have $S^2(t)>0$ for all  $t\in [l,2l]$.
 
If $2l>T-a$, according to \eqref{S11} we get respectively  
\begin{eqnarray*}
dS^2(t)&=& \frac{\mu}{T-a}\left(\int_{l-a}^{t-a}S^1(s)ds\right)+g(S^1(t-b))S^2(s)dW(s), \; t\in[l,T-a], S^2(l)=S^(l)
\end{eqnarray*}
and 
\begin{eqnarray*}
S^2(t)&=& S^1(T-a)+\mu\left(\int_{l-a}^{t-a}S^1(s)ds\right)+\int_{l}^{t}g(S^2(s-b))S^2(s)dW(s),\;\; t\in[T-a,2l].
 \end{eqnarray*}

Use again the same method as above we have
\begin{eqnarray*}
S^2(t)&=& S^1(l)\exp\left(\int_{l}^{t}g(S^1(s-b))dW(s)+\frac{1}{2}\int_{l}^{t}g^2(S^1(s-b))ds\right)\\
 &&+\left(\int_{l}^{t}\exp\left(\int_{s}^{t}g(S^1(r-b))dW(r)+\frac{1}{2}\int_{s}^{t}g^2(S^1(r-b))dr\right)\psi^2(s)ds\right){\bf 1}_{[l,T-a]}(t)\\
 &&+\left(\int_{l}^{t}\exp\left(\int_{s}^{t}g(S^1(r-b))dW(r)+\frac{1}{2}\int_{s}^{t}g^2(S^1(r-b))dr\right)S^1(s)ds\right){\bf 1}_{[T-a,2l]}(t).
\end{eqnarray*}
which together with the fact that $S^1$ is positive a.s. implies $S^2$ is positive.

Since $0<T<+\infty$, it is well know that there exists an integer $0<p<n$ such that $pl<T-a\leq (p+1)l$ and $nl<T\leq (n+1)l$. We define recursively  $(S^k)_{k=1}^{n+1}$ on $[2l,3l],\cdots,[pl,T-a],[T-a, (p+1)l],\cdots,[nl,T]$ the solution of \eqref{T1} as follows. For $k=1,\cdots, p$,
\begin{eqnarray*} 
S^k(t)&=& S^{k-1}(l)\exp\left(\int_{(k-1)l}^{t}g(S^{k-1}(s-b))dW(s)+\frac{1}{2}\int_{(k-1)l}^{t}g^2(S^{k-1}(s-b))ds\right)\\
 &&+\int_{(k-1)l}^{t}\exp\left(\int_{s}^{t}g(S^{k-1}(r-b))dW(r)+\frac{1}{2}\int_{s}^{t}g^2(S^{k-1}(r-b))dr\right)\psi^k(s)ds,\;\;\;\; t\in [(k-1)l,kl],
\end{eqnarray*}
 where 
 \begin{eqnarray*}
\psi^k(s)=\frac{\mu}{T-a}\int_{(k-1)l-a}^{s-a}S^{k-1}(r)dr
\end{eqnarray*}
and
\begin{eqnarray*} 
\widetilde{S}^p(t)&=& S^{p}(l)\exp\left(\int_{pl}^{t}g(S^{p}(s-b))dW(s)+\frac{1}{2}\int_{pl}^{t}g^2(S^{p}(s-b))ds\right)\\
 &&+\int_{pl}^{t}\exp\left(\int_{s}^{t}g(S^{p}(r-b))dW(r)+\frac{1}{2}\int_{s}^{t}g^2(S^{p}(r-b))dr\right)\widetilde{\psi}^p(s)ds,\;\; t\in [pl,T-a],
\end{eqnarray*}
 where 
 \begin{eqnarray*}
\widetilde{\psi}^p(s)=\frac{\mu}{T-a}\int_{pl-a}^{s-a}S^{p}(r)dr.
\end{eqnarray*}

We have also
\begin{eqnarray*} 
S^{p+1}(t)&=& \widetilde{S}^p(T-a)\exp\left(\int_{T-a}^{t}g(\widetilde{S}^p(s-b))dW(s)+\frac{1}{2}\int_{T-a}^{t}g^2(\widetilde{S}^p(s-b))ds\right)\\
 &&+\mu\int_{T-a}^{t}\exp\left(\int_{s}^{t}g(\widetilde{S}^p(r-b))dW(r)+\frac{1}{2}\int_{s}^{t}g^2(\widetilde{S}^p(r-b))dr\right)\widetilde{S}^p(s)ds,\;\; t\in [T-a,(p+1)l],
\end{eqnarray*}
\begin{eqnarray*} 
S^k(t)&=& S^{k-1}(l)\exp\left(\int_{(k-1)l}^{t}g(S^{k-1}(s-b))dW(s)+\frac{1}{2}\int_{(k-1)l}^{t}g^2(S^{k-1}(s-b))ds\right)\\
 &&+\mu\int_{(k-1)l}^{t}\exp\left(\int_{s}^{t}g(S^{k-1}(r-b))dW(r)+\frac{1}{2}\int_{s}^{t}g^2(S^{k-1}(r-b))dr\right)S^{k-1}(s)ds,\;\; t\in [(k-1)l,kl],
\end{eqnarray*}
for $k=p+2,\cdots,n$
and 
\begin{eqnarray*} 
S^{n+1}(t)&=& S^{n}(l)\exp\left(\int_{nl}^{t}g(S^{n}(s-b))dW(s)+\frac{1}{2}\int_{nl}^{t}g^2(S^{n}(s-b))ds\right)\\
 &&+\mu\int_{nl}^{t}\exp\left(\int_{s}^{t}g(S^{n}(r-b))dW(r)+\frac{1}{2}\int_{s}^{t}g^2(S^{n}(r-b))dr\right)S^{n}(s)ds,\;\; t\in [nl,T].
\end{eqnarray*}
By a similar argument, it follows that $S^{k}, k=1,\cdots, n+1$ and $\widetilde{S}^{p}$ are almost surely positive.

Let consider the processes $(S(t))_{t\in[0,T]}$ and $(\psi(t))_{t\in[0,T]}$ define respectively by 
\begin{eqnarray*}
	S(t)&=&\sum_{k=1}^{p}S^{k}(t){\bf 1}_{[(k-1)l,kl]}(t)+S^{p+1}(t){\bf 1}_{[pl,T-a]}(t)+S^{p+1}(t){\bf 1}_{[T-a, (p+1)l]}(t)\\
	 &&+\sum_{k=p+2}^{n}S^{k}(t){\bf 1}_{[(k-1)l,kl]}(t)+S^{n+1}(t){\bf 1}_{[nl,T]}(t),
\end{eqnarray*}
and
\begin{eqnarray*}
\psi(t)=\sum_{k=1}^{p}\psi^{k}(t){\bf 1}_{[(k-1)l,kl]}(t)+\widetilde{\psi}^{p}(t){\bf 1}_{[pl,T-a]}(t).
\end{eqnarray*} 

Finally $S(t)>0$ for all $t\in [0,T]$ a.s. Moreover it not difficult to  derive \eqref{S1} which end the proof.
 \end{proof}
\subsection{Case 2: $\displaystyle f(t,s_t)=\int_{-T}^{-a}\mu(s+u)s(t+u)\alpha(du)s(t)$}

Let us consider $f$ defined by 
\begin{eqnarray*}
f(t,s_t)=\int_{-T}^{-a}\mu(s+u)s(t+u)\alpha(du)s(t).	
\end{eqnarray*}
Therefore the generalized delayed SDE \eqref{GDSDE} becomes 
\begin{eqnarray}\label{T2}
	S(t)&=&\varphi(0)+\int_{0}^{t}\int_{-T}^{-a}\mu(s+u)S(s+u)\alpha(du)S(s)ds+\int_{0}^{t}g(S(s-b))S(s)dW(s),\ t\in[0,T] \nonumber\\
  S(t)&=&\varphi(t),\ t\in[-T,0].
\end{eqnarray} 
 We have this result.
\begin{theorem} \label{Theorem 2}
Assume $({\bf A1})$-$({\bf A3})$. If $\varphi(0)>0$, then a unique solution of GDSDE  \eqref{T2} is positive. Moreover, setting $l=\min(a,b)$, we obtain,   for $t\in[0,T]$,
\begin{eqnarray*}
S(t)=\varphi(0)\exp\left(\mu\int_{0}^{t}\Psi(s)ds-\frac{1}{2}\int_{0}^{t}g^{2}(S(s-b))ds +\int_{0}^{t}g(S(s-b))) dW(s)\right),	
\end{eqnarray*}
  where 
\begin{eqnarray*}
\Psi(s)=\sum_{k=1}^{p}\left(\int_{(k-1)l-a}^{s-a}S(u)\alpha(du)\right){\bf 1}_{[(k-1)l,kl]}(s)+\left(\int_{pl-a}^{s-a}S(u)\alpha(du)\right){\bf 1}_{[pl,T]}(s),
\end{eqnarray*}
with $p\in \N^{*}$ such that $pl\leq T<(p+1)l$
\end{theorem}
\begin{proof}
Let set
\begin{eqnarray*}
I(t)&=&\int_{0}^{t}\int_{-T}^{-a}\mu(s+u)S(s+u)\alpha(du)S(s)ds
\end{eqnarray*} 
Since $\alpha$ is a translation invariant measure on $[-T,0]$, we have
\begin{eqnarray*}
I(t)&=&\int_{0}^{t}\int_{s-T}^{s-a}\mu(r)S(r)\alpha(dr)S(s)ds,\\
&=& \int_{0}^{t}\left(\int_{s-T}^{-a}\mu(r)\varphi(r)\alpha(dr)+\mu\int_{-a}^{s-a}S(r)\alpha(dr)\right)S(s)ds.
\end{eqnarray*} 
Since $\mu(t)=0, t\in [-T,-a]$, we have
\begin{eqnarray}\label{W}
I(t)&=& \int_{0}^{t}\Psi(s)S(s)ds.
\end{eqnarray}
According to \eqref{T2}, il follows from \eqref{W} that
\begin{eqnarray}\label{S2}
	S(t)=\varphi(0)+\mu\int_{0}^{t}\Psi(s)S(s)ds+\int_{0}^{t}g(S(s-b))S(s)dW(s).
\end{eqnarray}
For $t\in [0,l]$, we have $t-a\in [-T,0]$ and  $t-b\in [-T,0]$.  Next for all $s\in [0,t],\; S(s-b)=\varphi(s-b)$ and $S(s)=\varphi(s)$ for all $s\in [-a,t-a]$.
Therefore 
\begin{eqnarray*}
S^1(t)=\varphi(0)+\mu\int_{0}^{t}\Psi(s)S(s)ds+\int_{0}^{t}g(\varphi(s-b))S(s)dW(s),
\end{eqnarray*}
where
\begin{eqnarray*}
	\Psi^1(s)=\mu\int^{s-a}_{-a}\varphi(r)\alpha(dr).
\end{eqnarray*}
Next using Girsanov theorem, we have for all $t\in [0,b]$,
\begin{eqnarray*}
S^1(t)=\varphi(0)\exp\left(\mu\int_{0}^{t}\Psi^1(s)ds-\frac{1}{2}\int_{0}^{t}g^2(\varphi(s-b))ds +\int_{0}^{t}g(\varphi(s-b))dW(s)\right).
\end{eqnarray*}
Therefore as $\varphi(0)>0$ we have $S^1(t)>0$. the rest of this proof is similar to one appear in the proof of Theorem \ref{Theorem 1} so we omit it.  
 \end{proof}
  
 \begin{example}
 Assume $({\bf A5})$. Then GDSDE  \eqref{T2} becomes:
 \begin{eqnarray*}
\begin{cases}
S(t)=\varphi(0)+\frac{\mu}{T-a}\displaystyle\int_{0}^{t}\left(\int_{-a}^{s-a}S(r)dr\right)S(s)ds+\displaystyle\int_{0}^{t}g(S(s-b)) S(s)dW(s), \ t\in[0,T] \\\\
  S(t)=\varphi(0),\ t\in[-T,0].
 \end{cases} 
 \end{eqnarray*}
Moreover, we get  
\begin{eqnarray*}
 S(t)=\varphi(0)\exp\left(\frac{\mu}{T-a}\int_{0}^{t} \Psi(s) ds-\frac{1}{2}\int_{0}^{t}g^{2}(S(s-b))ds+\int_{0}^{t}g(S(s-b))dW(s)\right).
\end{eqnarray*}
 where 
\begin{eqnarray*}
\Psi(s)=\sum_{k=1}^{p}\left(\int_{(k-1)l-a}^{s-a}S(u)du\right){\bf 1}_{[(k-1)l,kl]}(s)+\left(\int_{pl-a}^{s-a}S(u)du\right){\bf 1}_{[pl,T]}(s),
\end{eqnarray*}
with $p\in \N^{*}$ such that $pl\leq T<(p+1)l$.
 \end{example}

\section{A option price model with a delayed risk and risk-free asset}
This section is devoted to evaluate the price of European options with an underlying asset describe by a generalized delayed stochastic differential equations (GDSDEs, short) study in Section 2.

Roughly speaking, let consider a simple financial market containing a risk-free asset (a bond or bank account) $(B(t))_{t\geq 0}$ and single stock whose price $(S(t))_{t\geq 0}$ satisfies respectively GDSDE \eqref{T1} and \eqref{T2}. There are no transaction costs and the underlying stock pays no dividends in this market. We consider an option, written on the stock, with maturity $\xi$ and terminal time $T$. Our main objective is to derive the fair price of the option at each time $t\in [0,T]$. We shall consider two king of market.

\subsection{Financial risk and risk-free asset market I}
In this market, the free asset is described by $B(t)=e^{r\, t}$ and the single stock satisfies the following GDSDE
\begin{eqnarray*}
S(t)&=&\varphi(0)+\mu\int_{0}^{t}\int_{-T}^{-a}\mu(s+u)S(s+u)\alpha(du)S(s)ds+\int_{0}^{t}g(S(s-b))S(s)dW(s).	
\end{eqnarray*}
Assuming $\alpha$ satisfies $({\bf A5})$, hence the above GDSDE is equivalent to 
 \begin{eqnarray}\label{T21}
S(t)=\varphi(0)+\frac{\mu}{T-a}\int_{0}^{t}\Psi(s)S(s)ds+\int_{0}^{t}g(S(s-b))S(s)dW(s),
 \end{eqnarray}
 where
 \begin{eqnarray*}
\Psi(s)=\sum_{k=1}^{p}\left(\int_{(k-1)l-a}^{s-a}S(u)du\right){\bf 1}_{[(k-1)l,kl]}(s)+\left(\int_{pl-a}^{s-a}S(u)du\right){\bf 1}_{[pl,T]}(s),
\end{eqnarray*}
with $p\in \N^{*}$ such that $pl\leq T<(p+1)l$.
First, we will obtain an equivalent martingale
measure with the help of Girsanov's theorem. For this let 
\begin{eqnarray*}
\widetilde{S}(t)=\frac{S(t)}{B(t)}=e^{-rt}S(t)	
\end{eqnarray*}
and 
be the discounted stock price process.
 \begin{proposition}\label{P1}
 Assume $({\bf A2})$-$({\bf A5})$. Then we derive
 \begin{eqnarray}\label{M1}
 \widetilde{S}(t)&=& \varphi(0)+\int^t_0 \widetilde{S}(s)g(S(s-b))d\widehat{W}(s),
 \end{eqnarray}
 where 
  \begin{eqnarray}\label{G0}
 	\widehat{W}(t)=W(s)+\int^t_0\left(\frac{\mu\Psi(s)-r(T-a)}{(T-a)g(S(s-b))}\right)ds.
 \end{eqnarray}
Moreover, denoting by 
\begin{eqnarray*}
\rho(T)=\exp\left[-\int_{0}^{T}	\frac{\mu\Psi(s)-r(T-a)}{(T-a)g(S(s-b))}dW(s)-\frac{1}{2}\int_{0}^{T}\left(\frac{\mu\Psi(s)-r(T-a)}{(T-a)g(S(s-b))}\right)^2ds\right],
\end{eqnarray*} 
 $(\widehat{W}(t))_{t\geq 0}$ is a standard Wiener process under the measure $\Q$ defined by $d\Q=\rho(T)d\P$.
 \end{proposition}

 \begin{proof}
  It follows from Itô formula applied to $e^{-rt}S(t)$ that 
 \begin{eqnarray*}
 d\widetilde{S}(t)  &=& e^{-rt}dS(t)-re^{-rt}S(t)dt\\
 &=&e^{-rt}(\mu\,\Psi(t)S(t)dt+g(S(t-b))S(t)dW(t))-re^{-rt}S(t)dt\\
 &=&\widetilde{S}(t)\left[\frac{\mu\,\Psi(t)-r(T-a)}{(T-a)}dt+ g(S(t-b))dW(t)\right].
 \end{eqnarray*}
 According to $({\bf A4})$ and since $\widetilde{S}(0)=\varphi(0)$ we have 
 \begin{eqnarray*}
 \widetilde{S}(t)&=&\varphi(0)+\int^t_0 \widetilde{S}(s)g(S(s-b))d\widehat{W}(s).
 \end{eqnarray*}
 On the other hand, let set
 \begin{eqnarray*}
 \theta(s)=-\frac{\mu\Psi(s)-r(T-a)}{(T-a)g(S(s-b))}.
 \end{eqnarray*}
 Since $S(t)>0$ a.s. for all $t\in[0,T]$ and in view of $({\bf A4})$, the process $\theta$ is well defined and predictable. Moreover, the almost sure boundedness of the process $S$ on $[0,T]$, due to it sample-path continuity on $[0,T]$, and use again $({\bf A4})$ implies that $\theta$ is almost surely bounded on $[0,T]$. Hence
 \begin{eqnarray*}
 	\int^T_0|\theta(s)|^2ds<+\infty.
 \end{eqnarray*}
 Finally using the similar argument used in \cite{Aal} we obtain
 \begin{eqnarray*}
 	\E_{\P}\left[ \exp\left(\int_{0}^{t}\theta(s)dW(s)-\frac{1}{2}\int_{0}^{t}|\theta(s)|^2ds\right)\right]=1.
 \end{eqnarray*}
 Therefore, in view of Girsanov's theorem, the process $\widehat{W}$ defined by \eqref{G0} is a standard Wiener process under the measure $\mathbb{Q} $ defined by $d\Q=\rho(T)d\P$. 
 \end{proof}
 \begin{remark}\label{OP}
 In view of \eqref{M1} the discounted stock price $\widetilde{S}$ is a continuous $\Q$-local martingale. In other words, $\Q$ is
an equivalent local martingale measure. Thus, (see Theorem 7.1 in  \cite{KK})) the market containing $(B,S)$ on $[0,T]$ satisfies the no-arbitrage property which states that there is no admissible self-financing strategy which gives an arbitrage opportunity.  	
 \end{remark}
Let now derive the completeness of the market $(B,S)$ on $[0,T]$. For this let $\xi$ be a positive integrable $\mathcal{F}^{S}_{T}$-measurable payoff of an option associated to the stock and set
\begin{eqnarray*}
	M(t)=\E_{\Q}(e^{-rT}\xi|\mathcal{F}^{S}_{t})=\E_{\Q}(e^{-rT}\xi|\mathcal{F}^{\widehat{W}}_{t}).
\end{eqnarray*}

 \begin{theorem}\label{theorem 20}
  Assume $({\bf A2})$-$({\bf A5})$. Then there is an adapted and square integrable process $\{h(s),\ s\in [0,T]\}$ such that 
   \begin{eqnarray*}
    M(t) =\mathbb{E}_{\mathbb{Q}}(e^{-rT}\xi) +\int_{0}^{t} h(s)d\widehat{W}(s)	
   \end{eqnarray*}
  and the hedging strategy is given by
    \begin{eqnarray*} 
       \pi_{S}(t)=\frac{h(t)}{\tilde{S}(t)g(S(t-b))},\ \pi_{B}(t)=M(t)-\pi_{S}(t)\widetilde{S}(t), \ t\in [0,T]. 
   \end{eqnarray*}
Moreover at any time $t\in [0,T]$, the fair price $V(t)$ of the option is given by the formula 
  \begin{eqnarray}\label{pricing}
  	V (t)=e^{-r(T-t)}\E_{\mathbb{Q}}(\xi|\mathcal{F}^{S}_{t}).
\end{eqnarray}
\end{theorem}
\begin{proof}
In virtue of definition of $W,\; \widehat{W},\;\widetilde{S}$ it is not hard to observe that $\mathcal{F}^{S}=\mathcal{F}^{\tilde{S}}=\mathcal{F}^{\widehat{W}}=\mathcal{F}^{W}$. Let consider an integrable non-negative $\mathcal{F}^{S}_T$-measurable random variable $\xi$ be a contingent claim and define 
\begin{eqnarray*}
	M(t)=\E_{\Q}(e^{-rT}\xi|\mathcal{F}^{S}_t)=\E_{\Q}(e^{-rT}\xi|\mathcal{F}^{\widehat{W}}_t),\;\;\; t\in[0,T].
\end{eqnarray*}
The martingale representation theorem implies that there exists an $\mathcal{F}^{\widehat{W}}$-predictable process $\{h(t),\; \in [0,T]\}$ such that
\begin{eqnarray*}
	\int^{T}_{0}h^2(t)dt<+\infty
\end{eqnarray*}
and 
\begin{eqnarray*}
	M(t)=\E_{\Q}(e^{-rT}\xi)+\int_{0}^{t}h(s)d\widehat{W}(s), \; \;\;  t\in [0,T].
\end{eqnarray*}
Let consider the strategy $\{\pi_{B}(t),\pi_{S}(t) : t \in [0,T]\}$ which consists in holding $ \pi_{S}(t)$ units of the stock and $\pi_{B}(t)$ units of the bond at time $t$ and define by 
\begin{eqnarray}\label{Strat}
	\pi_{S}(t)=\frac{h(t)}{\widetilde{S}(t)g(S(t-b)}\;\;\;\mbox{and}\;\;\; \pi_{B}(t)=M(t)-\pi_{S}(t)\widetilde{S}(t).
\end{eqnarray}
According its definition and together with \eqref{Strat}, the value of the portfolio at any time $t\in [0,T]$ is given by 
 \begin{eqnarray}\label{Port}
 	V (t)&=& \pi_{B}(t)e^{rt} +\pi_{S}(t)S(t)\nonumber\\
 	&=& e^{rt}M(t).
 \end{eqnarray}
  Therefore, it follows from the product rule and \eqref{Port} that 
  \begin{eqnarray*}
  dV(t)&=& e^{rt}dM(t) + M(t)d(e^{rt})\\
   &=&\pi_{B}(t)d(e^{rt})+\pi_{S}dS(t), \ t \in [0,T].	
  \end{eqnarray*}
  Thus, $\{(\pi_{B}(t),\pi_{S}(t)) : t \in [0,T]\}$ is a self-financing strategy. On the other hand, 
  \begin{eqnarray*}
  	V(T)=e^{rT}M(T)=\xi,\;\;\;\; \mbox{a.s},
  \end{eqnarray*}
which means that the contingent claim $\xi$ is feasible. Finally, recalling \eqref{Port}, we get \eqref{pricing}.	
\end{proof}
  
The following result is a consequence of Theorem \ref{theorem 20}. It gives Black-Scholes type formula for the value of a European option on the stock at any time prior to maturity. For that, let $\phi$ be the distribution function of the standard normal law i.e 
\begin{eqnarray*}
	\Phi(x)=\frac{1}{\sqrt{2\pi}}\int_{-\infty}^{x}e^{-u^2/2}dx,\;\;\; x\in\R
\end{eqnarray*}
and set $\ell=min(a,b)$.

\begin{corollary}\label{Cor}
Assume $({\bf A2})$-$({\bf A5})$. Let $V(t)$ be the fair price
of a European call option written on the stock $S$ satisfied \eqref{T21} with exercise price $K$ and maturity time $T$. Then we have for $t\in [T-\ell,T]$ 
\begin{eqnarray}\label{16}
  V(t)=S(t)\Phi(\beta_{+}(t))-Ke^{-r(T-t)}\Phi(\beta_{-}(t)),
  \end{eqnarray}
   where
\begin{eqnarray*}  
  \beta_{\pm}(t)=\frac{\ln\left(\frac{S(t)}{K}\right)+\int_{t}^{T}(r\pm\frac{1}{2}g^{2}(S(s-b)))ds}{\sqrt{\int_{t}^{T}g^{2}(S(s-b))ds}}.
  \end{eqnarray*}
If $t<T-\ell$, then
\begin{eqnarray}\label{17}
V(t)=\E_{\Q}\left[\widetilde{S}(T-\ell)\Phi(a_1)-e^{-r(T-t)}K\Phi(a_2)|\mathcal{F}^{\widehat{W}}_t\right],
\end{eqnarray} 
where 
\begin{eqnarray*}
	a_1&=&\frac{1}{\sqrt{\int^T_{T-\ell}g^2(S(s-b))ds}}\left[\ln\left(\frac{S(t)}{K}\right)+r(T-t)+\frac{1}{2}\int_{T-\ell}^Tg^2(S(s-b))ds\right],\\
	a_2&=&\frac{1}{\sqrt{\int^T_{T-\ell}g^2(S(s-b))ds}}\left[\ln\left(\frac{S(t)}{K}\right)+r(T-t)-\frac{1}{2}\int_{T-\ell}^Tg^2(S(s-b))ds\right].
\end{eqnarray*}

The hedging strategy is given by
\begin{eqnarray*}
	\pi_S(t)=\Phi(\beta_{+}),\;\;\; \pi_{B}=-Ke^{-r(T-t)}\Phi(\beta_{-}(t)),\;\;\; t\in[T-\ell,T].
\end{eqnarray*}
\end{corollary}
\begin{proof}
Let consider European call option on the above market with a strike price $K$, maturity $T$. Since the playoff of such option is defined by $(S(T)-K)^{+}$. It follows from Theorem \ref{theorem 20} that the price of such option is given by: for $t \in [0,T]$,
\begin{eqnarray}\label{P2}
 V(t)& =& e^{-r(T-t)}\E_{\Q}((S(T)- K)^{+}|\mathcal{F}^{\widehat{W}}_{t})\nonumber\\
 &=& e^{-r(T-t)}\E_{\Q}(S(T)|\mathcal{F}^{\widehat{W}}_{t}){\bf 1}_{\{S(T)\geq K\}}- K e^{-r(T-t)}{\bf 1}_{\{S(T)\geq K\}}.
 \end{eqnarray}
 Recalling \eqref{M1}, we have 
\begin{eqnarray*}
S(T)= S(t)\exp\left(\int_{t}^{T}g((S(s-b))d\widehat{W}(s)+\int_{t}^{T}\left[r-\frac{1}{2}g^{2}(S(s-b))\right]ds\right),\ \forall t\in [0,T]. 
\end{eqnarray*}
 and then  for $t \in [0,T]$,
 \begin{eqnarray*}
 V(t)& =& e^{-r(T-t)}S(t)\E_{\Q}\left[\exp\left(\int_{t}^{T}g(S(s-b))d\widehat{W}(s)+\int_{t}^{T}\left[r-\frac{1}{2}g^{2}(S(s-b))\right]ds\right)|\mathcal{F}^{\widehat{W}}_{t}{\bf 1}_{\{S(T)\geq K\}}\right]\\
 &&- e^{-r(T-t)}K\P(S(T)\geq K).
 \end{eqnarray*} 
Since for $t\in [T-b,T],\, S(s-b)$ is $\mathcal{F}^{\widehat{W}}_{t}$-measurable for $s\in [0,t]$, then $\displaystyle -\frac{1}{2}\int_t^Tg^{2}S(s-b)ds$ is also $\mathcal{F}^{\widehat{W}}_{t}$-measurable. Furthermore, the process $\displaystyle \E\left(\int_{t}^{T}g(S(s-b))d\widehat{W}(s)|\mathcal{F}^{\widehat{W}}_{t}\right)$ is a centered Gaussian random variable with variance $\displaystyle\int_{t}^{T}g^{2}(S(s-b))ds$. On the other hand, $\{S(T)\geq K\}=\{X\leq \beta_{-}\}$ where $X$ is a Gaussian $\mathcal{N}(0,1)$-distributed random variable. Therefore, for all $t\in [T-\ell,T]$, 
\begin{eqnarray*}
 V(t)&=&\exp\left(-\frac{1}{2}\int_t^Tg^{2}(S(s-b))ds\right)S(t)\E_{\Q}\left[\exp\left(\sigma X\right){\bf 1}_{\{X\leq \beta_{-}\}}\right]-e^{-r(T-t)}K\P(X\leq \beta_{-}),
 \end{eqnarray*} 
 where
 \begin{eqnarray*}
 \sigma^2=\int_{t}^{T}g^{2}(S(s-b))ds.
\end{eqnarray*}
Finally, an elementary computation related to a standard normal distribution yields, for $t\in[T-\ell,T]$,
\begin{eqnarray*}
V(t)&=& S(t)\Phi(\beta_{+}(t))-e^{-r(T-t)}K\Phi(\beta_{-}(t)).
 \end{eqnarray*}
 Recalling again \eqref{M1}, we have
  \begin{eqnarray*}
 \widetilde{S}(T)=\widetilde{S}(T-\ell)\exp\left(\int^{T}_{T-\ell}g(S(u-b))d\widehat{W}(u)-\frac{1}{2}\int^{T}_{T-\ell}g^2(S(u-b))du\right).
  \end{eqnarray*}
 Therefore, in view of \eqref{P1} we obtain for $t<T-\ell$,
\begin{eqnarray*}
V(t)&=& e^{rt}\E_{\Q}\left[\widetilde{S}(T-\ell)\exp\left(\int_{T-\ell}^{T}g(S(s-b))d\widehat{W}(s)-\frac{1}{2}\int_{T-\ell}^{T}g^{2}(S(s-b))ds\right)|\mathcal{F}^{\widehat{W}}_{t}{\bf 1}_{\{S(T)\geq K\}}\right]\\
 &&- e^{-r(T-t)}K\P(S(T)\geq K)\\
 &=& e^{rt}\E_{\Q}\left[\widetilde{S}(T-\ell)\E_{\Q}\left[\exp\left(\int_{T-\ell}^{T}g(S(s-b))d\widehat{W}(s)-\frac{1}{2}\int_{T-\ell}^{T}g^{2}(S(s-b))ds\right)|\mathcal{F}^{\widehat{W}}_{T-\ell}{\bf 1}_{\{S(T)\geq K\}}\right]|\mathcal{F}^{\widehat{W}}_{t}\right]\\
 &&- e^{-r(T-t)}K\P(S(T)\geq K)
 \end{eqnarray*}
Set $B=\E_{\Q}\left[\exp\left(\int_{T-\ell}^{T}g(S(s-b))d\widehat{W}(s)-\frac{1}{2}\int_{T-\ell}^{T}g^{2}(S(s-b))ds\right)|\mathcal{F}^{\widehat{W}}_{T-\ell}{\bf 1}_{\{S(T)\geq K\}}\right]$, we have the following the same argument as above, 
 \begin{eqnarray*}
 	B=\Phi(a_1)
 \end{eqnarray*} 
 Therefore we have
 \begin{eqnarray*}
 	V(t)=\E_{\Q}\left(\widetilde{S}(T-\ell)\Phi(a_1)-e^{-(T-t)r}\Phi(a_2)|\mathcal{F}^{\widehat{W}}_{t}\right)
 \end{eqnarray*} 
 To calculate the hedging strategy for $t\in [T-\ell,T]$, it suffices to
use an idea from \cite{BR} (see pp. 95-96).
\end{proof}
\begin{remark}
If $g(x)=1$ for all $x\in \R_{+}$ then Equation \eqref{16} reduces
to the delayed Black and Scholes formula. Note that, in contrast with
non-delayed Black and Scholes formula, the fair price $V(t)$
in a general delayed model considered in Theorem \ref{Cor} depends not only
on the stock price $S(t)$ at the present time $t$, but also on the whole
segment $\{S(v),\; v\in [t-b,T-b]$.  (Of course $[t-b,T-b]\subset [ 0,t]$ since
$t\geq T-l$ and $l\leq b$.)
 \end{remark}   
\begin{remark}
Recalling that $a<T$, we have $T-\ell >0$. Therefore there exists $t\in [0,T]$ such that $t<T-\ell$. Therefore one can develop a recursive
procedure to calculate \eqref{17} by taking backwards steps of length $\ell$
from the maturity time $T$ of the option. Since the approach is similar to that stated by Arriojas et al. in \cite{Aal}, we will not explain it here. 
\end{remark}
\begin{remark}
According to Remark 4 in \cite{Aal}, one can rewrite the option
price $V(t),\; t\in [T-\ell,T]$ in terms of the solution of a random Black-
Scholes PDE of the form
\begin{eqnarray}\label{PDE}
\left\{
\begin{array}{l}
\frac{\partial F(t,x)}{\partial t}=-\frac{1}{2}g^2(S(t-b))x^2\frac{\partial^2 F(t,x)}{\partial x^2}-rx\frac{\partial F(t,x)}{\partial x}+rF(t,x),\;\; 0<t<T\\\\
F(T,x)=(x-K)^{+}, \;\; x>0.
\end{array}
\right.	
\end{eqnarray}
The above time-dependent random final-value problem admit a unique 
$(\mathcal{F}_{t\geq 0})_{t\geq 0}$-adapted random field $F(t,x)$. Using the classical Itô-Ventzell
formula (see, \cite{Kunita}) and \eqref{pricing} of Theorem \ref{theorem 20}, it can be shown that
\begin{eqnarray}\label{S}
	V(t)=e^{-r(T-t)}F(t,S(t)), \;\; t\in [T-\ell, T].
\end{eqnarray}
Moreove, the representation \eqref{S} do not much if $t\leq T-\ell$, because in this context, the solution $F$ of the Black-Scholes PDE \eqref{PDE} is
anticipating with respect to the filtration $(\mathcal{F}_{t\geq 0})_{t\geq 0}$. 
\end{remark}

\subsection{Financial risk and risk-free asset market II}
Let consider a market consisting of a single stock
whose price $S$ satisfies the SDDE \eqref{T2} and a riskless asset (a bond or bank account) $B$ satisfy the following dynamics
\begin{eqnarray}\label{Free}
B(t)&=& 1+\int_{0}^{t}\int_{-T}^{-a}r(s+u)B(s+u)\alpha(du)B(s)ds\nonumber\\
B(t)&=& 1,\;\; t\in [-T,0],
\end{eqnarray}
where 
\begin{eqnarray*}
 	r(t)=\left\{
 	\begin{array}{lll}
 	r&\mbox{if}& t\geq -a\\
 	0 &\mbox{if}& t< -a.	
 	\end{array}\right.
 \end{eqnarray*}

\begin{remark}
Let set $a=0$ and $\alpha=\delta_{-t}$ then 
\begin{eqnarray*}
\int_{-T}^{-a}r(t+u)B(t+u)\alpha(du)&=&\int_{-T}^{0}rB(t+u)\delta_{-t}(du)\\
	&=& r\, B(0)
\end{eqnarray*}
and equation \eqref{Free} coincide with classical ODE whose solution is $B(t)=e^{-rt}$.
\end{remark}
\begin{proposition}
Assume $({\bf A5})$. Then delayed ODE \eqref{Free} admit an unique positive solution. Moreover, we have 
\begin{eqnarray}\label{R2}
B(t)=\exp\left(\frac{r}{T-a}\int_{0}^{t}\Theta(s) ds\right)
\end{eqnarray}
where
\begin{eqnarray*}
\Theta(s)=\sum^{p}_{k=1}\left(\int_{(k-2)a}^{s-a}B(u)du\right){\bf 1}_{[(k-1)a,ka]}(s)+\left(\int_{(p-1)a}^{s-a}B(u)du\right){\bf 1}_{[pa,T]}(s),
\end{eqnarray*}
with $p\in \N^{*}$ such that $pa\leq T<(p+1)a$
\end{proposition}
We are now able to derive the explicit expression of the discounted stock price process. For this let set 
\begin{eqnarray*}
\rho(T)=\exp\left(-\int_{0}^{T}	\frac{\mu\Psi(s)-r\Theta(s)}{(T-a)g(S(s-b))}dW(s)-\frac{1}{2}\int_{0}^{T}\left(\frac{\mu\Psi(s)-r\Theta(s)}{(T-a)g(S(s-b))}\right)^2ds\right).
\end{eqnarray*} 
\begin{proposition}
Let $(\widetilde{S}(t))_{0\leq t\leq T}$ denote discounted stock price process defined by 
\begin{eqnarray*}
\widetilde{S}(t)=\frac{S(t)}{B(t)}.
\end{eqnarray*} 
We derive that
 \begin{eqnarray}\label{Mart1}
 \widetilde{S}(t)&=&S(0)+\int^t_0 \widetilde{S}(s)g(S(s-b))d\widehat{W}(s),
 \end{eqnarray}
 where 
  \begin{eqnarray}\label{Gir1}
 	\widehat{W}(t)=W(s)+\int^t_0\left(\frac{\mu\Psi(s)-r\Theta(s)}{(T-a)g(S(s-b))}\right)ds.
 \end{eqnarray}
Moreover, $(\widehat{W}(t))_{t\geq 0}$ and $(\widetilde{S}(t))_{0\leq t\leq T}$ are respectively a standard Wiener process and martingale under the measure $\Q$ defined by $d\Q=\rho(T)d\P$. 
 \end{proposition}
\begin{proof}
 In view of \eqref{R2}, we have 
 \begin{eqnarray*}
 	\widetilde{S}(t)=\exp\left(-\frac{r}{T-a}\int^t_0\Theta(s)ds\right)S(t).
 \end{eqnarray*}
 Therefore, it follows from Itô formula that 
 \begin{eqnarray*}
 d\widetilde{S}(t)&=& \exp\left(-\frac{r}{T-a}\int^t_0\Theta(s)ds\right)dS(t)-\frac{r\Theta(t)}{T-a}\exp\left(-\frac{r}{T-a}\int^t_0\Theta(s)ds\right)S(t)dt\\
 &=&\exp\left(-\frac{r}{T-a}\int^t_0\Theta(s)ds\right)\left[\frac{\mu}{T-a}\Psi(t)S(t)dt+g(S(t-b))S(t)dW(t))\right]\\
 &&-\frac{r\Theta(t)}{T-a}\exp\left(-\frac{r}{T-a}\int^t_0\Theta(s)ds\right)S(t)dt\\
 &=&g(S(t-b))\widetilde{S}(t)\left[\frac{\mu\Psi(t)-r\Theta(t)}{(T-a)g(S(t-b))}dt+ dW(t)\right].
 \end{eqnarray*}
The rest follows the same argument as in the proof of Proposition \ref{P1}.  
 \end{proof}
  \begin{remark}
 In view of \eqref{Mart1}, the discounted stock price $\widetilde{S}$ is a continuous $\Q$-local martingale. Therefore with the same reason as in Remark \ref{OP},  $\Q$ is an equivalent local martingale measure so that market containing $(B, S)$ on [0, T ] satisfies the no-arbitrage property.
 \end{remark}
We will now study the completeness of the financial market consisting of risk-free asset $B$ solution of delayed ODE \eqref{R2} and risky asset $S$ solution of GDSDE \eqref{T2} over the time interval $[0,T]$. The proof is almost identical to that of Theorem \ref{theorem 20}.

\begin{theorem}\label{theorem 2}
Assume $({\bf A2})$-$({\bf A5})$. Then price of the option with a positive $\mathcal{F}^{S}_{T}$-measurable and integrable payoff $\xi$ associated to the above stock is given by
  \begin{eqnarray}\label{Pricing}
  	V (t)=\exp\left(-\frac{r}{T-a}\int^T_t\Theta(s)ds\right)\E_{\mathbb{Q}}(\xi|\mathcal{F}^{S}_{t}).
  \end{eqnarray}
 Moreover, there exist an adapted and square integrable process $\{h(s),\ s\in [0,T]\}$ such that 
   \begin{eqnarray*}
    \E_{\Q}\left(\exp\left(-\frac{r}{T-a}\int^T_0\Theta(s)ds\right)\xi|\mathcal{F}^{S}\right)=\E_{\Q}\left(\exp\left(-\frac{r}{T-a}\int^T_0\Theta(s)ds\right)\xi\right) +\int_{0}^{t} h(s)d\widehat{W}(s).	
   \end{eqnarray*}
  Next, the hedging strategy is given by
    \begin{eqnarray*} 
       \pi_{S}(t)=\frac{h(t)}{\tilde{S}(t)g(S(t-b))},\ \pi_{B}(t)=M(t)-\pi_{S}(t)\widetilde{S}(t), \ t\in [0,T]. 
   \end{eqnarray*}
   \end{theorem}
  
\begin{corollary}
Assume $({\bf A2})$-$({\bf A5})$. Let denote $V(t)$, the price at $t$ of a European call option issued on stock $S$ with strike price $K$ and maturity $T$. Then we have: 
If $t\in [T-\ell]$, 
  \begin{eqnarray*}
  V(t)=S(t)\Phi(\beta_{+}(t))-\exp\left(-\frac{r}{T-a}\int^T_t\Theta(s)ds\right)K\Phi(\beta_{-}(t)),
  \end{eqnarray*}
 where 
 \begin{eqnarray*}
  \beta_{\pm}(t)=\frac{\ln\left(\frac{S(t)}{K}\right)+\displaystyle\int_{t}^{T}\left(\frac{r}{T-a}\Theta(s)\pm\frac{1}{2}g^{2}(S(s-b))\right)ds}{\sqrt{\displaystyle\int_{t}^{T}g^{2}(S(s-b))ds}}.
  \end{eqnarray*} 
 If $t<T-\ell$, then
\begin{eqnarray}\label{17}
V(t)=\E_{\Q}\left[\widetilde{S}(T-\ell)\Phi(a_1)-\exp\left(-\frac{r}{T-a}\int_t^T\Theta(s)ds\right) K\Phi(a_2)|\mathcal{F}^{\widehat{W}}_t\right],
\end{eqnarray} 
where 
\begin{eqnarray*}
	a_1&=&\frac{1}{\sqrt{\int^T_{T-\ell}g^2(S(s-b))ds}}\left[\ln\left(\frac{S(t)}{K}\right)+\int_{T-\ell}^T\left(\frac{r}{T-a}\Theta(s) +\frac{1}{2}g^2(S(s-b))\right)ds\right],\\
	a_2&=&\frac{1}{\sqrt{\int^T_{T-\ell}g^2(S(s-b))ds}}\left[\ln\left(\frac{S(t)}{K}\right)+\int_{T-\ell}^T\left(\frac{r}{T-a}\Theta(s)-\frac{1}{2}g^2(S(s-b))\right)ds\right].
\end{eqnarray*}
with $\Phi$ denotes the distribution function of the standard

normal law i.e 
\begin{eqnarray*}
	\Phi(x)=\frac{1}{\sqrt{2\pi}}\int_{-\infty}^{x}e^{-u^2/2}dx,\;\;\; x\in\R.
\end{eqnarray*}
The hedging strategy is given by
\begin{eqnarray*}
	\pi_S(t)=\Phi(\beta_{+}),\;\;\; \pi_{B}=-\exp\left(-\frac{r}{T-a}\int^T_t\Theta(s)ds\right)K\Phi(\beta_{-}(t)),\;\;\; t\in[T-l,T].
\end{eqnarray*}
\end{corollary}

\begin{remark}
Let suppose the stock $S$ satisfied \eqref{T1} and risk-free asset $B(t)=e^{-rt}$. Let $\widetilde{S}(t)$ denote a discounted stock price process defined by 
 \begin{eqnarray*}
\widetilde{S}(t)=\frac{S(t)}{B(t)}=e^{-rt}S(t), \;\; t\in[0,T].
\end{eqnarray*}
We have  
\begin{eqnarray*}
\widetilde{S}(t)=S(0)+\int^t_0g(S(s-b))\widetilde{S}(s)d\widehat{W}(s)+\frac{\mu}{T-a}\int_0^t \widetilde{\psi}(s)ds,	
\end{eqnarray*}
where
\begin{eqnarray*}
\widehat{W}(t)=	W(t)-\int_{0}^{t}\frac{r}{(T-a)g(S(s-b))}ds
\end{eqnarray*}
and
\begin{eqnarray*}
\widetilde{\psi}(t)=e^{-rt}\left(\int_{-a}^{t-a}\varphi(s)ds\right).
\end{eqnarray*}
Although $(\widehat{W}(t))_{t\geq t}$ is a standard Wiener process under the measure $\Q$ defined by $d\Q=\rho(T)d\P$, where
\begin{eqnarray*}
\rho(T)=\exp\left(-\int_{0}^{T}	\frac{\mu\,r}{(T-a)g(S(s-b))}dW(s)-\frac{1}{2}\int_{0}^{T}\left(\frac{\mu r}{(T-a)g(S(s-b))}\right)^2ds\right),
\end{eqnarray*}
the discounted process $(\widetilde{S}(t))_{t\geq 0}$ is not a $\Q$-martingale. Therefore, the market consisting of $(B(t),S(t))_{t\in[0,T]}$ do not satisfies the no-arbitrage property. In other words one have an admissible self-financing strategy which gives an arbitrage opportunity. 
\end{remark}


\begin{thebibliography}{KNU90}
\bibitem{Bachelier} Bachelier, L. 1900. Théorie de la Speculation. Annales de l'Ecole Normale Superieure 3. Gauthier-Villards, Paris.
\bibitem{BR} Baxter, M. and Rennie, A. Financial Calculus, Cambridge University Press,1996 .
\bibitem{BS} Black, F. and Scholes, M.. The pricing of options and corporate liabilities, Journal of Political Economy,  81 (May-June 1973), 637-654.
\bibitem{CA} : Clément Manga, Auguste Aman and Navègué Tuo : Asymptotic
behavior for delayed backward stochastic differential equations, Communications in Statistics -
Simulation and Computation, DOI: 10.1080/03610918.2023.2242011
 \bibitem{De} Delong, L. (2012). Applications of time-delayed backward stochastic differential equations to pricing, hedging and portfolio management. Applicationes Mathematicae, 39, 463-488 .
  \bibitem{HR} Hobson, D., and Rogers, L. C. G. Complete markets with stochastic volatility, Math. Finance, 8 (1998), 27-48 
  \bibitem{KK} Kallianpur, G. and Karandikar R. J. Introduction to Option Pricing Theory, Birkhuser Boston-Basel-Berlin, 2000.
  \bibitem{Kunita} Kunita, H. 1990. Stochastic Flows and Stochastic Differential Equations. Cambridge University Press, Cambridge, New York, Melbourne, Sydney.
  \bibitem{Lo1} Loewenstein, G. Anticipation and the valuation of delayed consumption. The Economic Journal 1987, 97, 666-684.
  \bibitem{Lo2} Loewenstein, G. and Prelec, D. (1993) Preferences for sequences of out-comes. Psychological Review 100, 91-108.
\bibitem{Aal} Mercedes Arriojas, M., Hu, Y. Mohammed, S.E. and Pap, G. A delayed Black and Scholes formula, Stochastic Analysis and Applications, 25 : 471-492, 2007 . 
 \bibitem{latexpratique} Mohammed, S. E. A. Stochastic Functional Differential Equations. Pitman 99 (1984).
\end{thebibliography}
\end{document}